\documentclass[]{article}

\usepackage[utf8]{inputenc}
\usepackage[english]{babel}
\usepackage{pdfpages}
\usepackage[a4paper]{geometry}
\usepackage{graphicx}
\usepackage{amsmath,amsthm,amssymb}
\usepackage{wasysym}
\usepackage{microtype}
\usepackage[colorlinks=false, pdfborder={0 0 0}]{hyperref}
\usepackage[capitalise]{cleveref}
\usepackage{romannum}
\usepackage[symbol]{footmisc}

\newcommand{\footremember}[2]{%
   \footnote{#2}
    \newcounter{#1}
    \setcounter{#1}{\value{footnote}}%
}
\newcommand{\footrecall}[1]{%
    \footnotemark[\value{#1}]%
}

\title{The Haagerup property and actions on von Neumann algebras}
\author{\stepcounter{footnote}Thiebout Delabie\footremember{alley}{Université de Neuchâtel, the authors are supported by grant 200021$\_$188578 of the Swiss National Fund for Scientific Research}, Alexandre Zumbrunnen\footrecall{alley}
}

\date{\today}

\newtheorem{theorem}{Theorem}[section]
\newtheorem{lemma}[theorem]{Lemma}
\newtheorem{proposition}[theorem]{Proposition}
\newtheorem{corollary}[theorem]{Corollary}
\newtheorem{definition}[theorem]{Definition}
\newtheorem{remark}[theorem]{Remark}
\newtheorem{example}[theorem]{Example}
\newtheorem{question}[theorem]{Question}

\newcommand{\N}{\mathbb{N}}
\newcommand{\Z}{\mathbb{Z}}

\newcommand{\R}{\mathbb{R}}
\newcommand{\C}{\mathbb{C}}
\newcommand{\E}{\mathbb{E}}
\newcommand{\T}{(\operatorname{T})}
\newcommand{\LL}{\operatorname{L}}
\newcommand{\Tr}{\operatorname{Tr}}
\newcommand{\B}{\operatorname{B}}
\newcommand{\rom}[1]{{\normalfont\Romannum{#1}}}

\begin{document}

\pagenumbering{arabic}

\maketitle

\begin{abstract}
In Chapter $2$ of "Groups with the Haagerup Property" \cite{cherix2012groups}, Jolissaint gives on the one hand a characterization of the Haagerup property in terms of strongly mixing actions on standard probability spaces; on the other hand he gives a noncommutative analogue of this result in terms of actions on factors. In the recent paper "A new characterization of the Haagerup property by actions on infinite measure spaces" \cite{delabie2019new}, the authors give a characterization of the Haagerup property but this time dealing with $C_0$-actions on infinite measure spaces. Following the spirit of this section, we give a noncommutative analogue in terms of $C_0$-actions on von Neumann algebras. 
Next we discuss some natural questions which remained open around $C_0$-dynamical systems. In particular we give examples of $C_0$-dynamical systems for groups acting properly on trees. Finally, we give a positive answer to the question of the ergodicity of such systems for non-periodic groups.
\end{abstract}

\section{Introduction}
The Haagerup property is a property of topological groups that has many different characterizations. Throughout this article, $G$ denotes a locally compact, second countable group which is assumed non-compact because the compact case is not relevant to this property \cite{cherix2012groups}.
Haagerup defined what is now called the H property in \cite{haagerup1978example}: motivated by the characterizations of nuclear $C^*$-algebras and injective von Neumann algebras in terms of existence of completely positive linear maps, and aiming to prove that complete positivity is essential in these results. 

Another characterization, first called a-$\T$-menability, was defined by Gromov in \cite{GromovATMenability}, as a strong negation for property $\T$.
It was later shown that these properties are equivalent.
For more information we refer to \cite{cherix2012groups}.

Recall that $G$ has the Haagerup property if and only if there exists a sequence of normalized, positive definite functions $(\varphi_n)_{n\ge 1}$ on $G$  such that $\varphi_n$ tends to the constant function $1$ uniformly on compact subsets of $G$ as $n \to \infty$, and each $\varphi_n \in C_0(G)$. In turn, $G$ is amenable if and only if each $\varphi_n$ can be chosen with compact support. Hence the Haagerup property is a weaker property than amenability, which itself generalizes both abelian and finite groups. Even though, there exists groups with the Haagerup property that are not amenable, free groups for example, there are a lot of properties of amenable groups that still hold for groups with the Haagerup property.
For example, groups with the Haagerup property satisfy the Baum-Connes conjecture \cite{higson2001theory}.

As von Neumann algebras are often seen as a non commutative version of measure spaces, since the space $\LL^\infty(\Omega,\mu)$ is a von Neumann algebra for every measure space $(\Omega,\mu)$, let us consider the following characterization of the Haagerup property :
\begin{theorem}[Theorem 2.1.3, \cite{cherix2012groups}]
	A locally compact second countable group has the Haagerup property if and only if it has a measure preserving action on a standard probability space that is strongly mixing and has a nontrivial asymptotically invariant sequence.
\end{theorem}
This can be generalized to include the non-commutative case.
This is done in Chapter 2 of \cite{cherix2012groups}, where theorem 2.1.5 states the following:
\begin{theorem}\label{noncommutative jol}
	Let $G$ be a locally compact second countable group. If there exists an action of $G$ on a von Neumann algebra $N$ with a faithful normal state $\varphi$ that is $G$-invariant, such that the action is strongly mixing and has a nontrivial asymptotically invariant sequence, then $G$ has the Haagerup property.
	\\
	Conversely, if $G$ has the Haagerup property, then it has such an action on the following von Neumann algebras:
	\begin{enumerate}
		\item $N = R$ (where $R$ is the hyperfinite \rom{2}$_1$ factor);
		\item $N = R\otimes \B(\ell^2)$;
		\item $N = R_\lambda$ (where $R_\lambda$ is the Powers factor).
	\end{enumerate}
\end{theorem}
Note that in the list above we have von Neumann algebras of types \rom{2}$_1$, \rom{2}$_\infty$ and \rom{3}.
\\
Similarly, we can consider Theorem 1.5 in \cite{delabie2019new}, which states the following (will be specified in \cref{THM DJZ}) :
\begin{theorem}\label{thm:Omega}
	A locally compact second countable group has the Haagerup property if and only if it has a measure preserving action on a measure space that is $C_0$ and has an almost invariant sequence.
\end{theorem}
In this article we provide a new characterization of the Haagerup property that is the non commutative version of the characterization in \cref{thm:Omega}.
We prove the following theorem:
\begin{theorem}\label{thm:main}
	Let $G$ be a locally compact second countable group and let $N$ be a von Neumann algebra. If there exists an action of $G$ on $N$ that is $C_0$ and has an almost invariant sequence, then $G$ has the Haagerup property.
	Conversely, if $G$ has the Haagerup property, then there exists von Neumann algebras $N$ of every type for which there exists an action of $G$ on $N$ that is $C_0$ and has an almost invariant sequence.
\end{theorem}
In section \ref{sec1} we show how \cref{thm:Omega} translates to the language of von Neumann algebras. That is, we show that a group has the Haagerup property if and only if it has an action on $\LL^\infty(\Omega,\mu)$ that is $C_0$ and has almost invariant vectors.
We also show that the existence of such an action implies the Haagerup property for any von Neumann algebra $N$.

In section \ref{sec2} we show that such an action exists for every group with the Haagerup property and every von Neumann algebra listed below:
\begin{enumerate}
		\item $N = \LL^\infty(\Omega,\mu)$ (where $(\Omega,\mu)$ is the measure space in \cref{thm:Omega});
		\item $N = \B(\LL^2(\Omega,\mu))$ (see \cref{operatoralgebra});
		\item When $G$ admits a discrete subgroup $\Gamma$ satisfying the hypothesis of \cref{crossedproduct}, $$N=\LL^\infty(\Omega,\mu)\rtimes \Gamma;$$
		\item When $G$ is unimodular, $N = \B(\LL^2(\Omega,\mu)) \rtimes G$ (see \cref{unimodular});
		\item Take $R_\lambda$, the Powers factor, $N = \B(\LL^2(\Omega,\mu)) \otimes R_\lambda$ (see \cref{prop:tensor}).
\end{enumerate}
In the last section, we discuss some natural questions which remained open related to \cref{thm:Omega}. When $G$ is amenable, then $L^{\infty}(G)$ has an invariant mean and it suffices to consider the regular representation of $G$ to have an example of \cref{thm:Omega}. So one natural question, first asked by Y.Stalder, is the following:
\begin{question}
What can be an explicit example of such a dynamical system $(\Omega,\mu)$ as in \cref{thm:Omega} when $G$ is not amenable ? 
\end{question}
In section \ref{sec3}, for groups which acts properly on trees, we find such action using tools from \cite{amine1} and \cite{amine2}.
In the case of the free group on two generators, we have :
\begin{example}
Let $\sigma:\mathbb{F}_2\curvearrowright T_4$ its action on its Cayley graph. Then an explicit example of infinite measure space for \cref{thm:Omega} is given by:
$$(\Omega(T_4)\times \R, \mu \otimes \lambda),$$
where $\Omega(T_4)$ is the space of all possible orientations of the tree and the action of $\mathbb{F}_2$ is a skew product action defined in section \ref{sec3}.1.
\end{example}
The second topic which is relevant in this context is the ergodicity of the $C_0$-dynamical system. \\
Let us recall that a locally compact group is periodic if every cyclic subgroup is relatively compact. So a locally compact group is non-periodic if and only if it contains a discrete subgroup isomorphic to $\Z$.\\
We proved under a weak assumption on the group, i.e. $G$ is non-periodic, that the action can be assumed to be ergodic. It results from the following proposition which is an extension of theorem $9.1.3$ of \cite{amine1} to the locally compact case:
\begin{proposition}
Let $G$ be a l.c.s.c. group which is non-periodic. Let $\pi: G\to \mathcal{O}(\mathcal{H})$ be a representation of class $C_0$ and $c\in Z^1(\pi,\mathcal{H})$ be an evanescent $1$-cocycle that is not a coboundary. Then the skew product action $\beta$ defined in section \ref{sec3}.1 is weakly mixing.
\end{proposition}

\subparagraph*{Acknowledgments}
The second author expresses his deep gratitude to Amine Marrakchi for pointing him many tools from \cite{amine1} and for his precious help in general concerning section $4$. We are also grateful to Alain Valette and Paul Jolissaint for their numerous relevant comments.

\section{From actions on von Neumann algebras to the Haagerup property}\label{sec1}
Throughout this section, $G$ denotes a non-compact locally compact, second countable group. In this article we will use two characterizations of the Haagerup property. For more informations on the Haagerup property we refer to \cite{cherix2012groups}. The first characterization comes from \cite{GromovATMenability} section $4.5.C$:
\begin{definition}\label{Haagerp Definition}
$G$ has the Haagerup property if and only if there exists a $C_0$ unitary representation $(\pi, \mathcal{H})$ of $G$, that is, all the matrix coefficients $\varphi_{\xi,\eta} : g\mapsto \langle \pi(g)\xi, \eta\rangle$ belong to $C_0(G)$, which weakly contains the trivial representation $1_G$.
\end{definition} 
The second characterization was introduced in \cite{delabie2019new}, which contains a dynamical characterization of the Haagerup property in terms of actions on infinite measure spaces. 
\begin{theorem}\label{THM DJZ}
	Let $G$ be a locally compact, second countable group. Then it has the Haagerup property if and only if it admits an action by Borel automorphisms on a measure space $(\Omega,\mu)$ such that $\mu$ is $\sigma$-finite, $G$-invariant and satisfies the following two additional properties:
	\begin{enumerate}
		\item for all measurable subsets $A,B\subset \Omega$ such that $0<\mu(A),\mu(B)<\infty$, the function $g\mapsto \mu(gA\cap B)$ is $C_0$, i.e. $(\Omega,\mathcal{B}, \mu,G)$ is a $C_0$-dynamical system;
		\item There exists a sequence of subsets $(A_n)_{n\ge 1}\subset \mathcal{B}$ with measure $\mu(A_n)=1$ for all $n$ such that for all compact subsets $K$ of $G$ we have 
\begin{equation}\label{Property3}
\lim_{n\rightarrow \infty} \sup_{g\in K} \mu(gA_n\Delta A_n)=0.
\end{equation}
	\end{enumerate}
\end{theorem}
Note that for the second condition one has the equality $\mu(gA_n\Delta A_n) = \|\pi_{\Omega}(g)\chi_{A_n}-\chi_{A_n}\|^2$. Hence, the second condition is equivalent to the existence of such sets $A_n$ with
\[\lim_{n\rightarrow \infty} \sup_{g\in K} \langle\pi_{\Omega}(g)\chi_{A_n},\chi_{A_n}\rangle=1.\]
Now, \cref{THM DJZ} is equivalent to Theorem 1.5 of \cite{delabie2019new}, as the vectors $\xi_n$ in the proof are indicator functions, see \cite{delabie2019new} on bottom of page 11.
\\
Theorem \ref{THM DJZ} can be expressed in terms of actions on the commutative von Neumann algebra $L^{\infty}(\Omega)$. Moreover, it turns out that this theorem has a noncommutative analogue. This section is devoted to state it and prove how the noncommutative analogue implies the Haagerup property. In order to translate theorem \ref{THM DJZ} in terms of von Neumann algebras, we need some preliminaries.\\
\linebreak
Let $N$ be a von Neumann algebra with separable predual, and let $\varphi$ be a \emph{faithful normal semifinite weight} on $N$. A weight on $N$ is the linear extension of a map $\varphi: N^+\to [0,+\infty]$ satisfying the following conditions:
$$\varphi(x+y)=\varphi(x)+\varphi(y), \quad x,y\in N^+; \quad \varphi(\lambda x)=\lambda\varphi(x) \; \lambda \ge 0.$$
It is said to be semifinite if $$\mathfrak{p}_{\varphi}=\{x\in N^+:\varphi(x)<\infty\}$$ generates $N$ and faithful if $\varphi(x)\neq 0$ for every non-zero $x\in N^+$. For more details we refer to \cite{takesaki2013theory}. Let
$$\mathfrak{n}_{\varphi}=\left\{x\in N: \; \varphi(x^*x)<\infty\right\},$$ 
and we denote $L^2(N,\varphi)$ the completion of $\mathfrak{n}_{\varphi}$ with respect to the norm 
\[ \|x\|_{\varphi}^2=\varphi(x^*x) \quad x\in \mathfrak{n}_{\varphi}.\]
Thus by extending the left multiplication on $\mathfrak{n}_{\varphi}$ we obtain an action of $N$ on the Hilbert space $L^2(N,\varphi)$. We will suppose that $L^2(N,\varphi)$ is separable. Moreover, we want an action $\alpha\colon G \to  \text{Aut}(N)$ such that its extension to $\LL^2(N,\varphi)$ is a unitary representation. This prerequisite is satisfied with a few simple assumptions:
\begin{proposition}\label{representation}
	Let $G$ be an l.c.s.c.\ group with an action $\alpha$ on a von Neumann algebra $N$.
	The action $\alpha$ extends to a unitary representation on $\LL^2(N,\varphi)$ if and only if $\alpha$ preserves the weight $\varphi$, that is $\varphi\circ \alpha_g=\varphi$, and for every $x\in \mathfrak{n}_{\varphi}$ we have that the maps $g \mapsto \varphi(x^* \alpha_g(x))$ are measurable.
\end{proposition}
\begin{proof}
	It is clear that the given action extends to a unitary representation on $L^2(N,\varphi)$. This follows from the fact that $G$ acts by isometries on $\mathfrak{n}_{\varphi}$ and then it suffices to consider their extension. In other words, to obtain an unitary representation it suffices to set for every $g\in G$ $$\pi(g)(x)=\lim_{n\rightarrow \infty} \alpha_g(x_n) \quad x\in L^2(N,\varphi),$$
	where $(x_n)_{n\ge 1} \subset \mathfrak{n}_{\varphi}$ is any sequence such that $\|x_n-x\|_{\varphi}\xrightarrow{n\to \infty} 0.$   The strong continuity of $\pi$ follows from Lemma $A.6.2$ on page $309$ of \cite{bekka2008kazhdan}. Indeed we assumed that $L^2(N,\varphi)$ is separable and the following sequence of functions defined by 
	\begin{equation*}
	f_n\colon G \to \mathbb{C} \colon g  \mapsto \langle \pi(g)x_n, x_n\rangle_{\varphi}=\varphi(x_n^*\alpha_g(x_n))
	\end{equation*}
	where $(x_n)_{n\ge 1}\subset \mathfrak{n}_{\varphi}$, is a sequence of measurable functions. Therefore for every $\xi\in L^2(N,\varphi)$ fixed, we have that the map $g\mapsto \langle\pi(g)\xi, \xi\rangle_{\varphi}$ is measurable as it is the pointwise limit of the sequence $(f_n)_n$.
	
	Next, we show the converse. Note that $\varphi(x^* \alpha_g(x)) = \langle \pi_g(x),x \rangle$ for every $x\in \mathfrak{n}_{\varphi}$. As $\pi$ is strongly continuous, we have that the maps $g \mapsto \varphi(x^* \alpha_g(x))$ are continuous and hence measurable. So, it suffices to show that $\alpha$ preserves the weight $\varphi$. Suppose it is not preserved. As it is the linear extension of a map on the positive elements of $N$, we can take $x\in N$ and $g\in G$ such that $\varphi(x^*x)\neq \varphi\circ\alpha_g(x^*x)$. Without loss of generality, we may assume that $\varphi(x^*x)$ is finite. However, that implies that $\|x\|_{\varphi}^2\neq \varphi\circ\alpha_g(x^*x) = \|\alpha_g(x)\|_{\varphi}^2$, which contradicts the fact that $\pi$ is unitary.
\end{proof}

\begin{definition}\label{def:C0}
	Let $G$ be an l.c.s.c. group with an action $\alpha$ on a von Neumann algebra $N$.
	The action $\alpha$ is $C_0$ for $\varphi$ if its extension to $\LL^2(N,\varphi)$ is a unitary representation and 
	\begin{equation}\label{property1}
	\lim_{g\rightarrow \infty} \varphi(\alpha_g(x)y)=0 \quad \forall x,y \in \mathfrak{n}_{\varphi}.
	\end{equation}
	A sequence of projections $(e_k)_{k \ge 1}\subset \mathfrak{n}_{\varphi}$ is said to be an almost invariant sequence for $\alpha$ and $\varphi$ if $\|e_k\|_{\varphi}=1$ for all $k$ and 
	\begin{equation}\label{property2}
	\lim_{k\rightarrow \infty} \sup_{g\in K} \| \alpha_g(e_k)-e_k\|_{\varphi}=0,
	\end{equation}
	for all compact subsets $K$ of $G$.
\end{definition}
Next we show that if $\alpha$ is $C_0$ and $\mathfrak{n}_{\varphi}$ contains an almost invariant sequence, then $G$ has the Haagerup property.

\begin{proposition}\label{DJ}
	Assume that there exists an action $\alpha : G \to \text{Aut}(N)$ which is $C_0$ for $\varphi$ and that $\mathfrak{n}_{\varphi}$ contains an almost invariant sequence for $\alpha$ and $\varphi$. Then the representation defined in \cref{representation} weakly contains the trivial representation of $G$ and is $C_0$. In particular $G$ has the Haagerup property.
\end{proposition}
\begin{proof}
	The set $\mathfrak{n}_{\varphi}$ is total in $L^2(N,\varphi)$ by construction, and for all $x,y\in \mathfrak{n}_{\varphi}$
	$$\langle \pi(g) x , y\rangle_{\varphi}=\varphi(y^*\alpha_g(x))\rightarrow 0,$$
	as $g\rightarrow \infty$. This proves that $\pi$ is of class $C_0$. Moreover, $1_G\prec \pi$ by the existence of an almost invariant sequence.
\end{proof}
	We can now translate theorem \ref{THM DJZ} in terms of actions on von Neumann algebras.
\begin{corollary}\label{cor:L_infty}
	$G$ has the Haagerup property if and only if there exists a commutative algebra $L^{\infty}(\Omega,\mu)$ with its canonical semifinite, but infinite, trace $\tau$ and an action $\alpha$ of $G$ on this von Neumann algebra which is $C_0$ for $\tau$ and contains an almost invariant sequence for $\alpha$ and $\tau$.
\end{corollary} 
\begin{proof}
If $G$ has such an action on $L^{\infty}(\Omega)$, then $G$ has the Haagerup property by Theorem \ref{DJ}. Now, if $G$ has the Haagerup property, then by theorem \ref{THM DJZ} there exists a $C_0$ dynamical system $(\Omega,\mathcal{B},\mu,G)$ which almost has invariant vectors. 
Now take the commutative von Neumann algebra $N=L^{\infty}(\Omega)$ endowed with the canonical trace given by $\tau=\int_{\Omega} d\mu$. Then $L^2(N,\tau)=L^2(\Omega,\mu)$ is separable as $\mu$ is $\sigma$-finite. The measurable action of $G$ on $N$ is given by the permutation representation $\pi_{\Omega}$ which is invariant for $\tau$ as $\mu$ is $G$-invariant. Furthermore the condition $(1)$ in the theorem \ref{THM DJZ} implies that the function  $g\mapsto \tau(x^*\alpha_g(x))$ is measurable for every $x\in \mathfrak{n}_{\tau}$ fixed, as well as the action is $C_0$ for $\tau$. Finally from the family $(A_n)$ we consider the associated sequence of unit vectors $\xi_n=\chi_{A_n}$ in $N$. As it satisfies the property (\ref{Property3}), then this sequence is an almost invariant sequence for $\alpha$ and $\tau$.
\end{proof}
In the next section we construct such actions on noncommutative von Neumann algebras.

\section{Constructing actions on von Neumann Algebras}\label{sec2}
In this section $G$ will denote a locally compact, second countable group with the Haagerup property. This section consists in three parts. First, we begin by creating an action on the $\rom{1}_{\infty}$ factor $\B(\mathcal{H})$ endowed with the canonical trace $\Tr$ and we prove this action is $C_0$ for $\Tr$ and there is an almost invariant sequence. 
Next, we give some actions on crossed products both on $N\rtimes_{\alpha} G$ and on von Neumann algebras of the form $N\rtimes_\alpha \Gamma$, with $\Gamma$ a discrete group. Finally we investigate some actions on tensor products.
\subsection{Action on the bounded operators of a Hilbert space}
Let $G$ be a locally compact, second countable group with the Haagerup property. Then using definition \ref{Haagerp Definition} there exists a $C_0$ representation $(\pi,\mathcal{H})$ of $G$ which weakly contains the trivial representation. Let us recall that we can assume without lost of generality that $\mathcal{H}$ is separable. Note that we can use $\mathcal{H}=L^2(\Omega)$ as in Theorem \ref{THM DJZ}, with the permutation representation.\\
Thus consider the $\rom{1}_{\infty}$ factor $N=\B(\mathcal{H})$ endowed with the canonical trace
$$\Tr(T)=\sum_n \langle Te_n,e_n\rangle,$$
where $(e_n)_n$ is a Hilbert basis of $\mathcal{H}$, and $T\in \B(\mathcal{H})$. Let $\alpha$ be the action of $G$ on $N$ given by 
$$\alpha_g(T)=\pi(g)T\pi(g)^*.$$
Then we have that $\Tr$ is $\alpha$-invariant
\begin{align*}
\Tr(\alpha_g(T))&=\sum_n\langle \pi(g)T\pi(g)^* e_n, e_n\rangle\\
&=\sum_n\langle T\pi(g^{-1}) e_n, \pi(g^{-1})e_n\rangle.
\end{align*}
The trace being invariant by change of basis, therefore we have $\Tr(\alpha_g(T))=\Tr(T).$\\
The following result can be deduced from \cite{bekka2008kazhdan} page 294.
\begin{proposition}\label{IsomorphismHilbert}
	Let $\mathcal{H}$ be a separable Hilbert space. Then the following map $\Psi$ is an isomorphism of Hilbert spaces from $\mathcal{H} \otimes \overline{\mathcal{H}}$ to $\LL^2(\B(\mathcal{H}),\Tr) $ defined by 
	\[\Psi(\xi\otimes \xi')(\eta) = \langle \eta, \xi \rangle\xi'.
	\]
\end{proposition}

We are grateful to Alain Valette for the following remark. With \cref{IsomorphismHilbert} in mind, it is then easy to see that the tensor product $\pi \otimes \overline{\pi}$ is unitary equivalent to the adjoint representation of $\pi$. Therefore \cref{operatoralgebra} follows from two remarks:
\begin{enumerate}
\item Since $\pi$ is $C_0$, then $\pi \otimes \overline{\pi}$ is $C_0$.
\item Since $\pi$ and $\overline{\pi}$ admit almost invariant vectors, then $\pi \otimes \overline{\pi}$ also.
\end{enumerate}
Nevertheless for the sake of completeness we give a proof of \cref{operatoralgebra} which has the advantage to construct "by hand" an almost invariant sequence for the adjoint action of $G$. Moreover, it will complete and justify remark \ref{remarkexempleprojections}.\\
It still remains to prove that the action $\alpha$ is a measurable action.  
\begin{proposition}
	Fix $T\in \mathfrak{n}_{\Tr}$. Then the map form $G$ to $\C$ defined by $g\mapsto \Tr(T^*\alpha_g(T))$ is measurable.
\end{proposition}
\begin{proof}
	As we have $$g\mapsto \Tr(T^*\alpha_g(T))=\sum_k\langle \pi(g)T \pi(g)^*e_k,Te_k\rangle,$$
	it suffices to show that the maps $f_k\colon g\mapsto \langle \pi(g) T \pi(g)^*e_k,Te_k\rangle$ are measurable for all $k$. Using Cauchy-Schwarz inequality we have:
	\begin{align*}
	|f_k(g)-f_k(e)|&\le |f_k(g)-\langle \pi(g)Te_k,Te_k\rangle|+|\langle \pi(g)Te_k,Te_k\rangle-\langle Te_k, Te_k\rangle |\\
	&\le \|\pi(g)\|\|T\|\|(\pi(g)^*-\text{Id})e_k\|\|Te_k\| +|\langle \pi(g)Te_k,Te_k\rangle-\langle Te_k, Te_k\rangle |.
	\end{align*}
	So, as $\text{SOT}\lim_{g\rightarrow e}\pi(g)=\text{Id}$, then $f_k(g)\rightarrow_{g\rightarrow e} f_k(e)$. Thus for every $k$ we have that $f_k$ is continuous. In particular $f_k$ is measurable for every $k$.
	Hence, the map $g\mapsto \Tr(T^*\alpha_g(T))$ is measurable as it is the simple limit of measurable functions.
\end{proof}
\begin{proposition}\label{operatoralgebra}
	Let $G$ be an l.c.s.c.\ group with the Haagerup property. Then there exists an action of $G$ on $(\B(\mathcal{H}),\Tr)$, which is $C_0$ for the trace $\Tr$ and it has an almost invariant sequence.	
\end{proposition}
\begin{proof}
	We begin to show the existence of a sequence of almost invariant projections. As $G$ has the Haagerup property, there exists, by \cref{Haagerp Definition}, a sequence of unit vectors $(\xi_n)_n$ in $\mathcal{H}$ such that for every compact sets $K$ of $G$:
	$$\lim_{n\rightarrow \infty} \sup_{g\in K} \|\pi(g)\xi_n-\xi_n\|=0.$$
	Set $P_n=\Psi(\xi_n\otimes \xi_n)$. Then we have:
	\begin{align*}
	\Tr(P_n^*P_n)=\sum_k\langle P_ne_k,P_ne_k\rangle&=\sum_k\langle \langle e_k, \xi_n\rangle \xi_n,  \langle e_k, \xi_n\rangle \xi_n \rangle\\
	&=\sum_k | \langle e_k, \xi_n\rangle|^2=\|\xi_n\|^2=1,
	\end{align*}
	where we used Parseval's identity. Moreover
	\begin{align*}
	\alpha_g(P_n)(\eta)&=\pi(g)P_n\pi(g^{-1})(\eta)\\
	&=\langle \eta, \pi(g) \xi_n\rangle \pi(g) \xi_n\\
	&=\Psi(\pi(g)\xi_n \otimes \pi(g)\xi_n)(\eta).
	\end{align*}
	Therefore:
	\begin{align*}
	\|\alpha_g(P_n)-P_n\|_{\Tr}^2&=\|\pi(g)\xi_n \otimes \pi(g)\xi_n - \xi_n\otimes\xi_n\|^2\\
	&=\left(\|\pi(g)\xi_n\|^4-2\text{Re}\left(\langle\pi(g)\xi_n,\xi_n\rangle^2\right)+ \|\xi_n\|^4\right)\\
	&=2\left(1-\text{Re}\left(\langle\pi(g)\xi_n,\xi_n\rangle^2\right)\right).\\
	\end{align*}
	Hence for every compact sets $K$ of $G$ we have:
	$$\lim_{n\rightarrow \infty} \sup_{g\in K} \|\alpha_g(P_n)-P_n\|_{\Tr}=0.$$
	It remains to show that $\alpha$ is $C_0$ for $\Tr$. First take elements of the form $\xi_1\otimes \eta_1, \xi_2\otimes \eta_2\in \mathcal{H}\overline{\otimes} \mathcal{H}$ and consider their image by $\Psi$ denoted $T_{\xi_1,\eta_1},T_{\xi_2,\eta_2} \in \mathfrak{n}_{\Tr}$. By Cauchy-Schwarz inequality we have:
	\begin{align*}
	\Tr(T_{\xi_1,\eta_1}^* \alpha_g(T_{\xi_2,\eta_2}))&=\sum_k \langle \pi(g)T_{\xi_2,\eta_2}\pi(g)^*e_k, T_{\xi_1,\eta_1}e_k\rangle\\
	&=\sum_k \langle \pi(g)\langle\pi(g^{-1})e_k, \xi_2\rangle \eta_2, \langle e_k, \xi_1\rangle \eta_1\rangle\\
	&=\langle\pi(g)\eta_2,\eta_1\rangle \sum_k \langle \pi(g^{-1})e_k,\xi_2\rangle\langle e_k, \xi_1\rangle\\
	&\le \langle\pi(g)\eta_2,\eta_1\rangle \|\xi_1\|\|\pi(g)\xi_2\|.
	\end{align*}
	As $\pi$ is a unitary $C_0$-representation, then we have $$\lim_{g\rightarrow \infty}\Tr(T_{\xi_1,\eta_1}^* \alpha_g(T_{\xi_2,\eta_2}))=0.$$
	Now take $T_1,T_2\in \mathfrak{n}_{\Tr}$. By Proposition \ref{IsomorphismHilbert} there exists $\xi,\eta\in \mathcal{H}\overline{\otimes} \mathcal{H}$ which represent $T_1,T_2$. Moreover for $\varepsilon>0$ fixed, there exists sequences $(\xi_n),(\eta_n)\subset \mathcal{H}$ such that:
	\begin{enumerate}
		\item $\| \xi - \sum_{n=1}^{N_{\varepsilon}} e_n \otimes \xi_n \|\le \varepsilon;$
		\item $\| \eta - \sum_{n=1}^{N_{\varepsilon}} e_n \otimes \eta_n \|\le \varepsilon.$
	\end{enumerate}
	Thus
	\begin{align*}
	\Tr(T_1^*\alpha_g(T_2))&=\langle \Psi\left(\xi -\sum_{n=1}^{N_{\varepsilon}}e_n \otimes \xi_n\right), \alpha_g(\Psi(\eta))\rangle_{\Tr}+\langle \Psi\left(\sum_{n=1}^{N_{\varepsilon}} e_n \otimes \xi_n\right), \alpha_g(\Psi(\eta))\rangle_{\Tr}\\
	&=\langle \Psi\left(\xi-\sum_{n=1}^{N_{\varepsilon}}e_n \otimes \xi_n\right), \alpha_g(\Psi(\eta))\rangle_{\Tr} +\\
	&\langle \Psi\left(\sum_{n=1}^{N_{\varepsilon}} e_n \otimes \xi_n\right), \alpha_g\left(\Psi\left(\eta-\sum_{n=1}^{N_{\varepsilon}} e_n\otimes \eta_n\right)\right)\rangle_{\Tr}\\
	&+\langle \Psi\left(\sum_{n=1}^{N_{\varepsilon}} e_n\otimes \xi_n\right), \alpha_g\left(\Psi\left(\sum_{n=1}^{N_{\varepsilon}}e_n\otimes \eta_n\right)\right)\rangle_{\Tr}
	\end{align*}
	Hence by Cauchy Schwarz and using that $\alpha$ is $C_0$ for the elements $T_{\xi_i,\eta_i}$, there exists a compact set $K$ of $G$ such that for all $g\in G\setminus K$ :
	\begin{align*}
	\Tr(T_1^*\alpha_g(T_2))&\le \varepsilon \|T_2\|_{\Tr} + \varepsilon\sqrt{\|T_1\|_{\Tr}+\varepsilon} +\sum_{i,j=1}^{N_{\varepsilon}} \frac{\varepsilon}{N_{\varepsilon}^2}\\
	&\le \varepsilon(1+\|T_2\|_{\Tr}+\sqrt{\|T_1\|_{\Tr}+\varepsilon}).
	\end{align*}
\end{proof}
\begin{remark}\label{remarkexempleprojections}
	In the context of Theorem \ref{THM DJZ}, this construction gives a $C_0$-action on $(\B(\LL^2(\Omega,\mu)),\Tr)$ relatively to the permutation representation $\pi_{\Omega}$. Moreover an almost invariant sequence of projections is given by the family $(A_n)_n\subset \mathcal{B}$ considering $$P_n(f)=\frac{\langle f,\chi_{A_n}\rangle}{\mu(A_n)}\chi_{A_n},$$
	for any $f\in L^2(\Omega).$
\end{remark}

\subsection{Actions on crossed products}
We begin this subsection with a brief recall on crossed products. Let $N$ be a von Neumann algebra acting on $\mathcal{H}$. Let $\Gamma$ be a discrete group and consider $\alpha \colon \Gamma \to \text{Aut}(N)$ an action of $\Gamma$ on $N$. We denote by $\mathcal{K}=l^2(\Gamma,\mathcal{H})$ the Hilbert space of all square integrable $\mathcal{H}$-valued functions. Now take the normal representation $\pi_{\alpha}$ of $N$ defined by 
$$(\pi_{\alpha}(x)\xi)(s)=\alpha_s^{-1}(x)\xi(s),\; \xi\in \mathcal{K}, \; x\in N,\; s\in \Gamma.$$
Moreover consider the unitary representation $\lambda$ of $\Gamma$:
$$(\lambda(t)\xi)(s)=\xi(t^{-1}s),\;  \xi\in \mathcal{K}, \; s,t\in \Gamma,$$
which is covariant with $\pi_{\alpha}$, that is $$\lambda(s)\pi_{\alpha}(x)\lambda(s)^*=\pi_{\alpha}\circ \alpha_s(x),\; x\in N, \; s\in \Gamma.$$
\begin{definition}
	The von Neumann algebra generated by $\pi_{\alpha}(N)$ and $\lambda(\Gamma)$ on $\mathcal{K}$ is called the crossed product of $N$ by $\alpha$ and denoted by $N\rtimes_{\alpha} \Gamma$.
\end{definition}
To each $x\in N \rtimes_{\alpha} \Gamma$ one can associate a function $x: \Gamma \to N$ with the idea that $x$ is represented by $\sum \pi_{\alpha}(x(g))\lambda(g)$. It is well known that this function uniquely determines $x$ (\cite{takesaki2013theory2}, p.366, \cite{pedersen1979c}, p.284). However the convergence of the series does not take place in any of the usual topologies on $N\rtimes_{\alpha} \Gamma$ (\cite{mercer1985convergence}).\\
Let $\varphi$ a faithful normal semi-finite weight on $N$ and consider $\hat{\varphi}$ the canonical induced weight on $N\rtimes_{\alpha} \Gamma$:
$$\hat{\varphi}(x)=\varphi(x(e)),\; x\in N\rtimes_{\alpha} \Gamma.$$
Furthermore we suppose that $\varphi$ is $\alpha$-invariant.
\begin{proposition}[Corollary $4$, \cite{mercer1985convergence}]
	For every $x\in N\rtimes_{\alpha} \Gamma$ we can find a unique $N$ valued function on $\Gamma$, also denoted by $x$, such that
	$\sum_{s\in \Gamma} \pi_{\alpha}(x(s))\lambda(s)$ converges in $L^2(N\rtimes_{\alpha} \Gamma, \hat{\varphi})$ to $x$.
\end{proposition}
Recall that by "$\sum_{s\in \Gamma} \pi_{\alpha}(x(s))\lambda(s)$ converges in topology $T$" we mean that the net $$\{\sum_{g\in F} \pi_{\alpha}(x(g))\lambda(g)\}_{F}$$ on the directed set of all finite subsets $F$ of $\Gamma$ converges in the topology $T$. For any $x\in N\rtimes_{\alpha} \Gamma$ we denote $x_F=\sum_{g\in F} \pi_{\alpha}(x(g))\lambda(g)$. Then applying Proposition $3$ of \cite{mercer1985convergence} to $\hat{\varphi} \circ \pi_{\alpha}$ one has:
\begin{equation}\label{convergenceL2}
(\hat{\varphi}(xx^*)-\hat{\varphi}(x_Fx_F^*))\rightarrow 0.
\end{equation}
\begin{lemma}\label{summable}
	The function $x$ of $\Gamma$ in $N$ representing $x\in \mathfrak{n}_{\hat{\varphi}}$ is $\LL^2$-square summable.
\end{lemma}
\begin{proof}
	For any finite set $F\subset \Gamma$ one has :
	\begin{align*}
	\|x_F\|_{\hat{\varphi}}^2=\hat{\varphi}(x_F^*x_F)&=\hat{\varphi}\left(\sum_{s,t\in F} \pi_{\alpha}(\alpha_s^{-1}(x(s)^*))\lambda(s)^*\pi_{\alpha}(x(t))\lambda(t)\right)\\
	&=\hat{\varphi}\left(\sum_{s,t\in F} \pi_{\alpha}(\alpha_s^{-1}(x(s)^*x(t)))\lambda(s^{-1}t)\right)\\
	&=\varphi\left(\sum_{s\in F} \alpha_{s^{-1}}(x(s)^*x(s))\right)\\
	&=\sum_{s\in F} \varphi(\alpha_{s^{-1}}(x(s)^*x(s)))\\
	&=\sum_{s\in F} \|x(s)\|_{\varphi}^2.
	\end{align*}
We conclude the proof using the normality of the weight and equation (\ref{convergenceL2}).
\end{proof}
Now suppose that $G$ has a $C_0$ action with an almost invariant sequence on $(N,\varphi)$, denoted $\beta$, which commutes with $\alpha$, that is $\alpha_s \circ \beta_g=\beta_g \circ \alpha_s$ for every $s\in \Gamma$, $g\in G$. Thus we can define an action $\iota$ of $G$ on $N\rtimes_{\alpha} \Gamma$ by 
$$\iota_g(x)=\sum_{s\in \Gamma} \pi_{\alpha}(\beta_g(x(s)))\lambda(s),\; x\in N\rtimes_{\alpha} \Gamma.$$
Then $\hat{\varphi}$ is $\iota$-invariant as $\varphi$ is $\beta$-invariant. Moreover this action is measurable in the following sense:
\begin{proposition}
	For every $x\in \mathfrak{n}_{\hat{\varphi}}$ fixed, the map from $G$ to $\C$ defined by $g\mapsto \hat{\varphi}(x^*\iota_g(x))$ is measurable.
\end{proposition}
\begin{remark}
	To prove this proposition we need an observation. Let $F$ be a finite subset of $\Gamma$ and consider an element in $N$ of the form  $$\sum_{s\in F} x(s)y(s).$$
	Then we have that:
	\begin{equation}\label{normality}
	\varphi\left(\sum_{s\in F} x(s)y(s)\right)=\sum_{s\in F} \varphi(x(s)y(s)).
	\end{equation}
	As $\varphi\left(\sum_{s\in F} x(s)y(s)\right)$ can be derived from:
	\begin{align*}
	\varphi\left( \sum_{s\in F} (x(s) + y(s))^*(x(s)+y(s))\right),\; \varphi\left( \sum_{s\in F}  y(s)^*y(s)\right),\; \varphi\left( \sum_{s\in F} x(s)^*x(s)\right),
	\end{align*}
 we can prove that the real and the imaginary parts of the left hand side are equal to those in the right hand side in (\ref{normality}). 

\end{remark}
\begin{proof}
	Let $x\in \mathfrak{n}_{\hat{\varphi}}$ fixed. Then we have for any finite $F\subset \Gamma$:
	\begin{align*}
	\hat{\varphi}(x_F^*\iota_g(x_F))&=\hat{\varphi}\left(\sum_{s,t\in F} \pi_{\alpha}(\alpha_s^{-1}(x(s)^*))\lambda(s)^* \pi_{\alpha}(\beta_g(x(t)))\lambda(t)\right)\\
	&=\hat{\varphi}\left(\sum_{s,t\in F} \pi_{\alpha}(\alpha_s^{-1}(x(s)^*)\beta_g(\alpha_s^{-1}(x(t))))\lambda(s^{-1}t)\right)\\
	&=\varphi\left(\sum_{s\in F} \alpha_s^{-1}(x(s))^* \beta_g(\alpha_s^{-1}(x(s)))\right)
	\end{align*}
	By the previous remark and equation (\ref{convergenceL2}), we have:
	$$\hat{\varphi}(x^*\iota_g(x))=\sum_{s\in \Gamma} \varphi\left(\alpha_s^{-1}(x(s))^* \beta_g(\alpha_s^{-1}(x(s)))\right).$$
	As the map $g\mapsto \varphi(y^*\beta_g(y))$ is measurable for every $y \in \mathfrak{n}_{\varphi}$ fixed, then $g\mapsto \hat{\varphi}(x^*\iota_g(x))$ is a simple limit of measurable functions. Therefore this map is measurable.
\end{proof}
\begin{proposition}\label{crossedproduct}
	Let $G$ be an l.c.s.c.\ group with the Haagerup property, let $\Gamma$ be a discrete group and let $N$ be a von Neumann algebra, with a faithful normal semi-finite weight $\varphi$, such that $G$ and $\Gamma$ have commuting actions $\beta$ and $\alpha$. Suppose $\beta$ is $C_0$ for $\varphi$ with an almost invariant sequence. Then $G$ has a $C_0$-action on $N\rtimes_\alpha \Gamma$ with an almost invariant sequence for $\hat{\varphi}$. 
\end{proposition}
\begin{proof} 
	First we prove that the action $\iota$ of $G$ on $N\rtimes_\alpha \Gamma$ is $C_0$ for $\hat{\varphi}$. Let $x,y\in \mathfrak{n}_{\hat{\varphi}}$ with $x=\sum_{s\in \Gamma} \pi_{\alpha}(x(s))\lambda(s)$ and $y=\sum_{t\in \Gamma} \pi_{\alpha}(y(t))\lambda(t)$. Then  using the previous remark and by Cauchy-Schwarz inequality, we have:
	\begin{align*}
	\hat{\varphi}(\iota_g(x)y)&=\sum_{s\in \Gamma} \varphi\left( \beta_g(x(s))\alpha_s^{-1}(y(s))\right)\\
	&\le \sum_{s\in \Gamma} \|x(s)\|_{\varphi} \|y(s)\|_{\varphi} \\
	&\le \frac{1}{2}\sum_{s\in \Gamma}\left( \|x(s)\|_{\varphi}^2+ \|y(s)\|_{\varphi}^2\right).
	\end{align*}
	Then by Lemma \ref{summable} we deduce $$\sum_{s\in \Gamma}\left( \|x(s)\|_{\varphi}^2+ \|y(s)\|_{\varphi}^2\right)<\infty.$$
	Thus for $\varepsilon>0$ fixed, there exists a finite subset $K$ of $\Gamma$ such that
	$$\sum_{s\in K^c} \left(\|x(s)\|_{\varphi}^2+ \|y(s)\|_{\varphi}^2\right)<\varepsilon.$$
	Moreover for every $h\in K$, as $\beta$ is a $C_0$-action  for $\varphi$ there exists a compact subset $Q_h\subset G$ such that
	$$\varphi(\beta_g(x(h))\alpha_h^{-1}(y(h)))<\frac{\varepsilon}{2|K|}\quad \forall g\in G\setminus Q_h.$$
	Then set $Q=\bigcup_{h\in K} Q_h$ which is a compact of $G$ such that for all $g\in G\setminus Q$ we have:
	\begin{align*}
	|\hat{\varphi}(\iota_g(x)y)|&\le \frac{1}{2}\sum_{s\in K^c} \left(\|x(s)\|_{\varphi}^2+ \|y(s)\|_{\varphi}^2\right) + \sum_{h\in K} |\varphi(\beta_g(x(h))\alpha_h^{-1}(y(h)))|\\
	&< \frac{\varepsilon}{2} + |K| \frac{\varepsilon}{2|K|}=\varepsilon.
	\end{align*}
	Now by assumption there exists an almost invariant sequence $(e_k)_{k\ge 1}$ for $\beta$ and $\varphi$. Hence considering the following sequence:
	$$x_n=\pi_{\alpha}(e_n)\lambda(e),$$
	we obtain an almost invariant sequence for $\iota$ and $\hat{\varphi}$ as for all $k$ we have
	$$\|\iota_g(x_k)-x_k\|_{\hat{\varphi}}=\|\beta_g(e_k)-e_k\|_{\varphi}.$$
\end{proof}
Let $G$ be a locally compact second countable group with the Haagerup property. Then using Corollary \ref{cor:L_infty} there exists a $C_0$-action of $G$ with almost invariant sequence on $L^{\infty}(\Omega)$. Now suppose there exists a countable group $\Gamma$ acting on $L^{\infty}(\Omega)$ and commuting with the $G$-action. Then using Proposition \ref{crossedproduct}  there exists a $C_0$-action of $G$ on the crossed product algebra $L^{\infty}(\Omega)\rtimes\Gamma$ with an almost invariant sequence. Moreover if we assume that the action of $\Gamma$ is free and ergodic, then we get an action with the desired properties on a von Neumann algebra of type \rom{2}$_{\infty}$. The same can be deduced for the von Neumann algebra $\B(\LL^2(\Omega))\rtimes \Gamma$. When we consider the trivial action of $\Gamma$ in those examples, we find the well-known fact ( see \cref{prop:tensor}) that $G$ has a $C_0$-action with almost invariant sequence on $L^{\infty}(\Omega)\otimes L(\Gamma)$ or $\B(\LL^2(\Omega))\otimes L(\Gamma)$, where $L(\Gamma)$ is the group von Neumann algebra of $\Gamma$.\\
\linebreak
Our next goal is to consider actions of $G$ on the crossed product $N\rtimes_{\alpha} G$, where $\alpha$ is a $C_0$-action with an almost invariant sequence in $N$. In this context we recall that the $W^*$-crossed product can be seen as the weak closure of the $*$--algebra of operators $K(G,N)$($\sigma$-weakly continuous bounded functions with compact supports). For each $y\in K(G,N)$ and $\xi\in L^2(G,\mathcal{H})$, considering the regular representation $(\pi_{\alpha} \times \lambda, L^2(G,\mathcal{H}))$, we have :
$$(((\pi_{\alpha}\times \lambda)y)\xi)(t)=\int_{G} (\pi_{\alpha}(y(s))\lambda(s)\xi)(t)ds=\int_{G} \pi(\alpha_{t^{-1}}(y(s)))\xi(s^{-1}t)ds.$$
Moreover we recall that the involution and convolution in $K(G,N)$ is defined by :
$$y^*(t)=\Delta(t)^{-1}\alpha_t(y(t^{-1})^*)$$
$$(y\times z)(t)=\int y(s) \alpha_s(z(s^{-1}t))ds,$$
for all $y,z\in K(G,N)$, where $\Delta$ is the modular function. We refer to \cite{pedersen1979c} page $278$ for more details about $W^*$-crossed products with locally compact groups. Consider a faithful normal semifinite weight $\varphi$ on $N$. Now using theorem $3.1$ of \cite{haagerup1979dual}, we have the dual weight $\tilde{\varphi}$ which is such that for $x\in K(G,N)$ we have :
$$\tilde{\varphi}((\pi_{\alpha} \times \lambda)(x^*\times x))=\varphi((x^*\times x)(e)).$$
Now assume that there exists a $C_0$-action of $G$ on $N$, denoted $\alpha$, with an almost invariant sequence for $\varphi$.  We can then consider the $W^*$-crossed product $N\rtimes_{\alpha} G$ with its faithful normal semifinite weight $\tilde{\varphi}$. We define an action $\iota$ of $G$ on $N\rtimes_{\alpha} G$ by :
$$\iota_g(x)=\Delta(g)^{-1}\lambda(g)x\lambda(g)^*, \; x\in N\rtimes_{\alpha} G.$$ 
For $x\in K(G,N)$ that is :
\begin{align*}
(\iota_g(x)\xi)(t)&=((\lambda(g)(\pi_{\alpha}\times \lambda)(x)\lambda(g)^*)\xi)(t)\\ 
&=\int \pi(\alpha_{t^{-1}g}(x(g^{-1}s)))\xi(gs^{-1}t)ds,\\
\end{align*}
which is the same as :
\begin{align*}
\left(\left(\int \lambda(g) \pi_{\alpha}(x(s)) \lambda(s) \lambda(g)^* ds\right)\xi\right)(t)=\int \pi(\alpha_{t^{-1}g}(x(s)))\xi(gs^{-1}g^{-1}t)ds,
\end{align*}
where we used the change of variables $u=g^{-1}s$.\\
Moreover using Theorem $3.1$ part $(c)$ and $(d)$ in \cite{haagerup1979dual} we have that
\begin{align*}
\tilde{\varphi}(\iota_g(x^*x))&=(\varphi \circ \pi_{\alpha}^{-1})\circ T( \iota_g(x^*x))\\
&=(\varphi \circ \pi_{\alpha}^{-1})(\lambda(g)\pi_{\alpha}((x^*x)(e))\lambda(g)^*)\\
&=\varphi(\alpha_g((x^*x)(e))),
\end{align*} 
for $x\in K(G,N)$.
Hence $\iota$ preserves $\tilde{\varphi}$.\\
\begin{proposition}\label{unimodular}
Let $G$ be an l.c.s.c. unimodular group with the Haagerup property and let $N$ be a von Neumann algebra, with a faithful normal semi-finite weight $\varphi$ such that $G$ has a $C_0$-action $\alpha$ with an almost invariant sequence. Then $G$ has a $C_0$-action on $N\rtimes_{\alpha} G$ with an almost invariant sequence for $\tilde{\varphi}$. 
\end{proposition}
\begin{remark} Since $G$ is unimodular then $\Delta\equiv 1.$ However in order to emphasize where the hypothesis of unimodularity is necessary in the following proof, it was decided to start this proof as if we consider any group l.c.s.c. 
\end{remark} 
\begin{proof}
First we prove that the action $\iota$ of $G$ on $N\rtimes_{\alpha} G$ is $C_0$. Let $x,y\in K(G,N)\cap \mathfrak{n}_{\tilde{\varphi}}$ and suppose without lost of generality that $\|x\|_{\tilde{\varphi}}=\|y\|_{\tilde{\varphi}}=1$. Then, recalling that $G$ is first assumed to be a general l.c.s.c. group, consider:
\begin{align*}
\tilde{\varphi}((x+\iota_g(y))^*(x+\iota_g(y)))&=\varphi\left((x+\iota_g(y))^*\times (x+\iota_g(y))(e)\right)\\
&=\varphi \left(\int (x+\iota_g(y))^*(s)\alpha_s((x+\iota_g(y))(s^{-1}))ds\right)\\
&=\varphi \left(\int \Delta(s)^{-1}\alpha_s((x+\iota_g(y))(s^{-1})^*)\alpha_s((x+\iota_g(y))(s^{-1}))ds\right)\\
\end{align*}
Hence we have that it is equal to :
$$\varphi \left(\int \Delta(s)^{-1}\alpha_s(x(s^{-1})^*+\Delta(g)^{-1}\lambda(g)y(s^{-1})^*)\lambda(g)^*)\alpha_s(x(s^{-1})+\Delta(g)^{-1}\lambda(g)y(s^{-1})\lambda(g)^*)ds\right).$$
Using the covariance relation and making the change of variables $u=s^{-1}$ we obtain :
$$\varphi \left(\int \left[ \alpha_{u^{-1}}(x(u)^*)+\Delta(g)^{-1}\alpha_{u^{-1}g}(y(u)^*)\right]\left[\alpha_{u^{-1}}(x(u))+\Delta(g)^{-1} \alpha_{u^{-1}g}(y(s))\right] du \right).$$
As $\varphi$ is normal, then we have :
\begin{align*}
\int \|x(u)+\Delta(g)^{-1}\alpha_{g}(y(u))\|_{\varphi}^2 du&=\int \|x(u)\|_{\varphi}^2 + 2 \text{Re}\left(\langle x(u),\Delta(g)^{-1}\alpha_g(y(u))\rangle_{\varphi}\right) + \Delta(g)^{-1}\|y(u)\|_{\varphi}^2 du.
\end{align*}
Now when $G$ is unimodular then $\Delta=1$ and therefore:
$$\Delta(g^{-1})\int\|y(u)\|_{\varphi}^2du=\tilde{\varphi}(y^*y).$$
So we have that :
$$\tilde{\varphi}((x+\iota_g(y))^*(x+\iota_g(y))=\tilde{\varphi}(x^*x)+\tilde{\varphi}(y^*y)+2\Delta(g)^{-1}\text{Re}\left( \int \varphi(x(u)^*\alpha_g(y(u))) du\right).$$
Using Cauchy-Schwarz inequality we have that :
$$\langle x(u), \alpha_g(y(u))\rangle_{\varphi}\le \|x(u)\|_{\varphi}\|y(u)\|_{\varphi}\le \frac{1}{2}(\|x(u)\|_{\varphi}^2+\|y(u)\|_{\varphi}^2).$$ 
As $x,y\in \mathfrak{n}_{\tilde{\varphi}}$, then their respective functions are square integrable ( this is the analogue of lemma \ref{summable} for locally compact groups). Hence we can use the dominated convergence theorem and the fact that $x,y$ have compact supports to conclude that :
$$\lim_{g\rightarrow \infty} 2\Delta(g)^{-1}\text{Re}\left( \int \varphi(x(u)^*\alpha_g(y(u))) du\right)=0,$$
as $\Delta(g)=1$ for all $g\in G$. It follows that $$\lim_{g\rightarrow \infty} \tilde{\varphi}(x\iota_g(y))=0\quad \forall x,y\in K(G,N),$$
which proves that $\iota$ is $C_0$ for $\tilde{\varphi}$.\\
Now by assumption  there exists an almost invariant sequence $(e_k)_{k\ge 1}$ for $\alpha$ and $\varphi$. Consider a continuous function $f:G\to \mathbb{C}$ with compact support and take 
$$x_n=e_nf.$$
Then we have :
\begin{align*}
\|\iota_g(x_n)-x_n\|_{\tilde{\varphi}}^2&=\int \|\alpha_g(x_n(u))-x_n(u)\|_{\varphi}^2du\\
&=\int \|\alpha_g(e_n)-e_n\|_{\varphi}^2 |f(u)|^2 du.
\end{align*}
Hence for each compact $K\subset G$ we have : $$\sup_{g\in K}\|\iota_g(x_n)-x_n\|_{\tilde{\varphi}}^2\xrightarrow{n\rightarrow \infty} 0.$$
\end{proof}

\subsection{Actions on tensor products}

\begin{proposition}\label{prop:tensor}
	Let $G$ be an l.c.s.c.\ group with the Haagerup property and let $N$ and $M$ be von Neumann algebras with faithful, normal, semi-finite weights $\varphi_N$ and $\varphi_M$ such that $G$ has a $C_0$-action on $N$ with an almost invariant sequence and $G$ has an action on $M$ that also has an almost invariant sequence and that extends to a unitary representation on $\LL^2(M,\varphi_M)$.
	Then $G$ has a $C_0$-action on $N\overline{\otimes} M$ with an almost invariant sequence.
\end{proposition}
\begin{proof}
	Let $\alpha$ be the action of $G$ on $N$ and $\beta$ be the action of $G$ on $M$. Take $\varphi=\varphi_N\otimes \varphi_M$. By Definition 4.2 and Proposition 4.3 of \cite{takesaki2013theory}, we know that $\varphi$ is a weight on $N\overline{\otimes} M$.
	We show that the action $\gamma=\alpha\otimes\beta$ of $G$ on $N\overline{\otimes}  M$ is $C_0$ and has an almost invariant sequence.
	\\
	First, note that $\LL^2(N\overline{\otimes}  M, \varphi_N\otimes \varphi_M) = \LL^2(N,\varphi_N)\overline{\otimes} \LL^2(M,\varphi_M)$. As both the extension of $\alpha$ and $\beta$ to $\LL^2(N,\varphi_N)$ and $\LL^2(M,\varphi_M)$ respectively are unitary representation, we have that $\gamma$ extends to a unitary representation as well. As $\alpha$ is $C_0$ for $\varphi_N$, then by proposition \ref{DJ} its extension is $C_0$. Since the tensor product of a $C_0$-representation with any representation is $C_0$, then $\gamma$ is $C_0$ for $\varphi$.
	\\
	Finally, by proposition \ref{DJ} both the extension of $\alpha$ and $\beta$ have almost invariant vectors. Since the tensor product of representations with almost invariant vectors has almost invariant vectors, then $\gamma$ has an almost invariant sequence.
\end{proof}

With this result we can construct many new examples of von Neumann algebras that have a $C_0$-action of $G$ with an almost invariant sequence.
From \cref{cor:L_infty} and \cref{operatoralgebra} we know that $G$ has such an action on a $\LL^\infty(\Omega)$ and a $\B(\mathcal{H})$. We also have from \cref{noncommutative jol} an action of $G$ on the hyper finite \rom{2}$_1$ factor $R$ with its trace $\tau$ and an action of $G$ on Powers factor $R_\lambda$ with the Powers state $\varphi_\lambda$. Now due to \cref{prop:tensor}, we know that $G$ has a $C_0$-action with an almost invariant sequence on $\LL^\infty(\Omega)\overline{\otimes} R$, $\B(\mathcal{H})\overline{\otimes} R$, $\LL^\infty(\Omega)\overline{\otimes} R_\lambda$ and $\B(\mathcal{H})\overline{\otimes} R_\lambda$.
Hence, we have an example of a type \rom{2}$_1$ von Neumann algebra, a \rom{2}$_\infty$ factor and type \rom{3} factor.
\section{Discussion on Ergodicity}\label{sec3}
Theorem \ref{THM DJZ} characterizes the Haagerup property of $G$ in terms of actions on infinite measure spaces. In this section we investigate if such actions can always be taken ergodic. First we will give an another proof of this theorem using a skew product action introduce in section 2.5 in \cite{amine1}. Via tools of \cite{amine1} and \cite{amine2} this new approach will allow us to assume that such dynamical system is ergodic. Secondly we give non trivial examples of such dynamical systems for groups acting properly on trees. In particular, it will give an example of a measure space with a $C_0$-action for the free group on two generators.\\
We begin by a brief recall of the Gaussian process. We refer to \cite{amine2} for more information. Let $\mathcal{H}$ be a real separable Hilbert space. Then it is already known that there exists a standard probability space $(\hat{\mathcal{H}},\mu)$ and a linear isometry $\iota : \mathcal{H}\hookrightarrow L^2(\hat{\mathcal{H}},\mu)$ such that :
\begin{enumerate}
\item[1)] For all $\xi \in \mathcal{H}$, $\iota(\xi)$ is a centered Gaussian random variable of variance $\|\xi\|^2$;
\item[2)] The family of random variables $(\iota(\xi))_{\xi \in \mathcal{H}}$ generates the $\sigma$-algebra of measurable subsets of $(\hat{\mathcal{H}},\mu)$.
\end{enumerate}
Moreover this random process is unique in the following sense. 
Suppose there exists another triple $(X,\nu, \iota')$ satisfying $(1)$ and $(2)$, then there exists a measurable bijection $\varphi : (\hat{\mathcal{H}},\mu)\to (X,\nu)$ such that $\varphi_*\mu=\nu$ and almost everywhere $$\iota'(\xi)\circ \varphi=\iota(\xi),$$
for all $\xi\in \mathcal{H}$. We use the same notation as \cite{amine1} and denote this unique Gaussian process $(\hat{\xi})_{\xi \in \mathcal{H}}$. Note that from the uniqueness of the Gaussian process, it follows that for an orthogonal transformation $U\in \mathcal{O}(\mathcal{H})$, there exists a measure preserving automorphism $\hat{U}$ of $(\hat{\mathcal{H}},\mu)$ such that $$\widehat{U\xi}=\hat{\xi}\circ \hat{U}^{-1}.$$
Now defining a new measure on $\hat{\mathcal{H}}$ by the formula $d\mu_{\eta}=\exp(-\frac{1}{2}\|\eta\|^2+\hat{\eta})d\mu$ for a fixed vector $\eta\in \mathcal{H}$, we can see that the random variable $$\hat{\xi}-\langle \xi,\eta\rangle$$ has a normal distribution $\mathcal{N}(0,\|\xi\|^2)$ with respect to $\mu_\eta$. Thus using uniqueness there exists a nonsingular automorphism $\hat{T}_{\eta}$ of $(\hat{\mathcal{H}},\mu)$ such that $$\hat{\xi}\circ \hat{T}_{\eta}^{-1}=\hat{\xi}-\langle \xi,\eta\rangle,$$ for all $\xi\in \mathcal{H}$. We will use the notation $\hat{T}_{\eta}(\omega)=\omega+\eta$ so that the intuitive formula $\langle \omega + \eta, \xi \rangle=\langle\omega, \xi\rangle + \langle \eta, \xi \rangle $ holds, where $\langle \omega, \xi\rangle $ denotes $\hat{\xi}(\omega)$. Then a simple computation shows that for $V\in \mathcal{O}(\mathcal{H})$ and $\eta\in \mathcal{H}$ we have $$\hat{V}(\omega + \eta)=\hat{V}(\omega)+V(\eta).$$ Therefore we get a morphism from $\text{Isom}(\mathcal{H})$ to $\text{Aut}(\hat{\mathcal{H}},[\mu])$. In particular, for every action  $\alpha:G\curvearrowright \mathcal{H}$ by affine isometries, we obtain a Gaussian action $\hat{\alpha}:G\curvearrowright (\hat{\mathcal{H}},\mu)$.\\
Now let $\pi: G\to \mathcal{O}(\mathcal{H})$ be an orthogonal representation of $G$ with the Haagerup property. Let $c\in Z^1(\pi,\mathcal{H})$ be a $1$-cocycle which is proper and take $\alpha: G\curvearrowright \mathcal{H}$ the associated proper affine isometric action. Consider the measurable space $\Omega=\hat{\mathcal{H}}\times \mathbb{R}$ endowed with the measure $\nu:=\mu \otimes \lambda$, where $\lambda$ denotes the Lebesgue measure. Then we define the skew product infinite measure preserving action $\beta: G\curvearrowright (\Omega,\nu)$ by the formula:
\begin{equation}\label{beta}
\beta_g(\omega,t)=(\hat{\pi}(g)w, t+\langle \omega, c(g^{-1})\rangle).
\end{equation}
First we prove that the Koopman representation associated to $\beta$ is of class $C_0$ using the following proposition taken from \cite{amine2} page $54$.
\begin{proposition}\label{Koopman}
Let $\sigma:G\curvearrowright X$ a nonsingular action on a space $X$ and $\pi:G\curvearrowright L^2(X)$ the associated Koopman representation. Then the following are equivalent: 
\begin{enumerate}
\item[(i)] $\lim_{g\to \infty} \langle \pi(g)\xi,\xi\rangle=0$ for some faithful $\xi \in L^2(X)^{+}$.
\item[(ii)] $\pi$ is mixing.
\end{enumerate}
\end{proposition}
Let $(g_n)_{n\ge 1}$ be a sequence in $G$ which tends to infinity. We will denote $X_n$ the random variable $\widehat{c(g_n)}$ with density $$\frac{1}{\|c(g_n)\|\sqrt{2\pi}}\exp\left(-\frac{x^2}{2\|c(g_n)\|^2}\right).$$
Let $\xi=1\otimes h\in L^2(\Omega,\nu)$, where $h(t)=\frac{1}{1+t^2}.$
Then we compute:
\begin{align*}
\lim_{n\to \infty}\langle \pi_{\Omega}(g_n)\xi,  \xi \rangle,
\end{align*}
where $\pi_{\Omega}$ is the Koopman representation associated to $\beta$. As $\nu$ is $G$-invariant, we have that $\pi_{\Omega}$ is simply the permutation representation on $L^2(\Omega)$. Hence:
\begin{align*}
\langle \pi_{\Omega}(g_n)\xi, \xi\rangle &=\int_{\Omega} (\pi_{\Omega}(g_n)\xi)(u) \xi(u) d\nu(u)\\
&=\int_{\hat{H}\times \R} \xi(\beta_{g_n^{-1}}(\omega, t))\xi(\omega,t) d(\mu \otimes \lambda)(\omega,t)\\
&=\int_{\hat{H}\times \R} \frac{1}{1+(t+X_n(\omega))^2}\frac{1}{1+t^2}d(\mu \otimes \lambda)(\omega,t).
\end{align*}
Using Fubini-Tonelli, we get:
\begin{align*}
\int_{\R}\frac{1}{1+t^2}\left(\int_{\hat{H}}\frac{1}{1+(t+X_n(\omega))^2} d\mu(\omega) \right)dt.
\end{align*}
Then we obtain:
\begin{align*}
\int_{\R} \frac{1}{1+t^2}\left(\frac{1}{\|c(g_n)\|\sqrt{2\pi}}\int_{\R} \frac{1}{1+(t+x)^2} e^{\frac{-x^2}{2\|c(g_n)\|^2}}  dx \right) dt.
\end{align*}
Applying again Fubini-Tonelli, this integral becomes : 
$$\frac{1}{\|c(g_n)\|\sqrt{2\pi}}\int_{\R\times \R} \frac{1}{1+t^2}\frac{1}{1+(t+x)^2} e^{\frac{-x^2}{2\|c(g_n)\|^2}}  d\lambda(x,t).$$
Now writing for all $n$ :
$$f_n(x,t)= \frac{1}{\|c(g_n)\|\sqrt{2\pi}}\frac{1}{1+t^2}\frac{1}{1+(t+x)^2} e^{\frac{-x^2}{2\|c(g_n)\|^2}},$$
we notice that $|f_n(x,t)|\le\frac{C}{(1+t^2)(1+(x+t)^2)} $ for every $(x,t)\in \R^2$ and for all $n$ large enough, where $C>0$. Indeed as the cocycle is proper then for $N$ large enough we get :$$\|c(g_n)\|\ge A> 0\; \forall n\ge N.$$
Moreover the sequence of functions $(f_n)$ converges pointwise to $0$, since $c$ is proper. Using the dominated convergence theorem, we deduce:
$$\lim_{n\to \infty} \langle \pi_{\Omega}(g_n)\xi,\xi\rangle=0.$$
As $\xi$ is a separating vector, we conclude with proposition \ref{Koopman} to obtain:
\begin{proposition}
Let $G$ be a locally compact second countable group which is a-$(T)$-menable. Then $\pi_{\Omega}$, the Koopman representation associated to $\beta$, is $C_0$.
\end{proposition}
\begin{remark}
It is interesting to notice that, with such an action, the $C_0$-property of $\pi_{\Omega}$ depends uniquely on the $1$-cocycle, and not on properties of the p.m.p. action $\hat{\pi}$ and therefore intrinsically from the orthogonal representation $\pi$.
\end{remark}
Moreover, we have that $\pi_{\Omega}$ admits a sequence of almost invariant vectors. Indeed consider for all $n$ $$\xi_n=1\otimes \frac{1}{\sqrt{2n}}\textbf{1}_{[-n,n]}.$$
Then :
\begin{align*}
\langle \pi_{\Omega}(g)\xi_n,\xi_n\rangle &=\frac{1}{2n}\frac{1}{\|c(g)\|\sqrt{2\pi}}\int_{\R}e^{\frac{-x^2}{2\|c(g)\|^2}} \left(\int_{\R} \textbf{1}_{[-n,n]}(t)\textbf{1}_{[-n,n]}(t+x) dt\right)dx\\
&=\frac{1}{2n}\frac{1}{\|c(g)\|\sqrt{2\pi}}\int_{0}^{\infty}e^{\frac{-x^2}{2\|c(g)\|^2}}(2n-x)dx +\frac{1}{2n}\frac{1}{\|c(g)\|\sqrt{2\pi}}\int_{-\infty}^0 e^{\frac{-x^2}{2\|c(g)\|^2}}(2n+x)dx\\
&=\frac{2(n+\frac{\|c(g)\|}{\sqrt{2\pi}})}{2n}.
\end{align*}
Thus by continuity of the $1$-cocycle, we deduce that for each compact subset $K$ of $G$, we have:
$$\sup_{g\in K}\langle \pi_{\Omega}(g)\xi_n,\xi_n\rangle\xrightarrow{n\to \infty} 1.$$
\\
Finally the continuity of $\pi_{\Omega}$ is obtained using lemma $A.6.2$ page 309 of \cite{bekka2008kazhdan}.\\
We improve here \cref{THM DJZ} (the improvement is the last statement):  
\begin{theorem}\label{THM DJZ 2}
	Let $G$ be a locally compact, second countable group. Then it has the Haagerup property if and only if it admits an action by Borel automorphisms on a measure space $(\Omega,\mu)$ such that $\mu$ is $\sigma$-finite, $G$-invariant and satisfies the following two additional properties:
	\begin{enumerate}
		\item for all measurable subsets $A,B\subset \Omega$ such that $0<\mu(A),\mu(B)<\infty$, the function $g\mapsto \mu(gA\cap B)$ is $C_0$, i.e. $(\Omega,\mathcal{B}, \mu,G)$ is a $C_0$-dynamical system;
		\item There exists a sequence of unit vectors $(\xi_n)_{n\ge 1}\subset L^2(\Omega,\mu)^+$ such that for all compact subsets $K$ of $G$, one has 
		$$
\lim_{n\rightarrow \infty} \sup_{g\in K} \langle \pi_{\Omega}(g)\xi_n,\xi_n\rangle=1.
$$
	\end{enumerate}
	Moreover if $G$ is non-periodic, then $C_0$-dynamical system may be assumed to be ergodic. 
\end{theorem}
So it remains to prove ergodicity. Actually the proof will be in two steps. First we will prove proposition \ref{weaklymixing} which is an extension of theorem 9.1.3 of \cite{amine1} to the locally compact case. Then we will see that a non-compact group $G$ with the Haagerup property always satisfies the hypothesis of this proposition. In order to do that we begin by some definitions and recalls which also fix the notations.\\
In \cite{amine2} the authors introduce an important notion to study affine isometric actions.
\begin{definition}
Let $\alpha:G\curvearrowright \mathcal{H}$ be an affine isometric action. Then the support of $\alpha$ is the linear subspace of $\mathcal{H}^0$ defined by 
$$\text{supp}(\alpha)=\bigcap_{\mathcal{K}\in E} \mathcal{K}^0$$
where $E$ is the set of all non-empty $\alpha$-invariant affine subspaces of $\mathcal{H}$. If $\text{supp}(\alpha)=\left\{0\right\}$, we say that $\alpha$ is evanescent.
\end{definition}
By proposition $2.10$ of  \cite{amine2}, we say that a $1$-cocycle $c\in Z^1(\pi,\mathcal{H})$ is evanescent if there exists a decreasing sequence of closed $\pi(G)$-invariant subspaces $H_n$ such that $\bigcap_n H_n=\left\{0\right\}$ and such that for every fixed $n$, the $1$-cocycle $c$ is cohomologous to a $1$-cocycle taking values in $H_n$.\\
Now from section $2.7$ of \cite{amine1} we recall a important result when we are interested to show ergodicity. Let $K\subset H$ be a closed subspace and $\widehat{\pi_{K}}:\hat{H}\to \hat{K}$ the natural measure preserving factor map. For every $\xi \in \mathcal{H}$, we denote by $\mathcal{M}(\xi+K)\subset L^{\infty}(\hat{H}\times \R)$ the von Neumann subalgebra of functions of the form
$$(\omega,t)\mapsto F(\widehat{\pi_{K}}(\omega), t+\langle \omega, \xi\rangle), \quad F\in L^{\infty}(\hat{K}\times \R).$$
Then :
\begin{proposition}{(Proposition 2.7, \cite{amine2})}\label{2.7 Amine}
Let $(\xi_i+H_i)_{i\in I}$ by a family of closed affine subspaces of a real Hilbert space $H$. Write $K=\bigcap_{i\in I} H_i$, $M=\bigcap_{i\in I}(\xi_i +H_i)$, $A=\bigcap_{i\in I} \mathcal{M}(\xi_i+H_i)$.
\begin{enumerate}
\item If $M=\emptyset$, then $A=L^{\infty}(\hat{K})\otimes 1.$
\item If $\eta \in M$, then $M=\eta +K$ and $A=\mathcal{M}(\eta +K)$.
\end{enumerate}
\end{proposition}
\begin{remark} Recall that a nonsingular action $\sigma: G \curvearrowright X$ is weakly mixing if the diagonal action $\sigma \otimes \rho : G\curvearrowright X \otimes Y$ is ergodic for every ergodic probability measure preserving(in short pmp) action $\rho:G \curvearrowright Y$.\\
Let us also recall that a locally compact group is periodic if every cyclic subgroup is relatively compact. So a locally compact group is non-periodic if and only if it contains a discrete subgroup isomorphic to $\Z$.
\end{remark} 

We can now state the main proposition which will be proved below:

\begin{proposition}\label{weaklymixing}
Let $G$ be a l.c.s.c. group which is non-periodic. Let $\pi: G\to \mathcal{O}(\mathcal{H})$ be a representation of class $C_0$ and $c\in Z^1(\pi,\mathcal{H})$ be an evanescent $1$-cocycle that is not a coboundary. Then the skew product action $\beta$ defined by (\ref{beta}) is weakly mixing.
\end{proposition}
We need one more result from \cite{amine2} which generalizes theorem $2.3$ of \cite{schmidt1982mildly}.
\begin{definition} Let $\pi: G\to \mathcal{O}(H)$ be an orthogonal representation. A countable subset $\Lambda \subset G$ is called mixing with respect to $\pi$ if for every $\xi,\eta \in H$, one has :
$$\lim_{g\in \Lambda, g\to \infty} \langle \pi(g)\xi, \eta\rangle=0.$$
\end{definition}
\begin{definition}Let $\sigma:G\curvearrowright X$ be a nonsingular action. A countable set $\Lambda \subset G$ is called recurrent with respect to $\sigma$ if for every $B\subset X$ with $B\neq 0$ there exists infinitely many $g\in \Lambda$ such that $gB\cap B\neq 0$.
\end{definition}
\begin{theorem}{(Theorem 7.13, \cite{amine2})}\label{7.13 Amine} Let $\sigma:G\curvearrowright X$ be a nonsingular action and $\rho:G\curvearrowright(Y,\nu)$ a pmp action. Suppose that there exists a subset $\Lambda \subset G$ such that $\Lambda$ is recurrent with respect to $\sigma$ and mixing with respect to $\rho$. Then every $\sigma \otimes \rho$-invariant function in $L^{\infty}(X\otimes Y)$ is contained in $L^{\infty}(X)$.
\end{theorem}
With this result we are able to prove \cref{weaklymixing}:
\begin{proof}[Proof of proposition \ref{weaklymixing}]
Let $\sigma:G \curvearrowright (Z,\rho)$ be any pmp action and consider the diagonal action $\gamma:G\curvearrowright (\hat{H}\times \R)\times Z$. First assume that $K\subset H$ is a closed $\pi(G)$-invariant subspace and that $c(g)\in K^{\perp}$ for all $g\in G$. Identifying $\hat{H}$ with $\hat{K}\times \hat{K}^{\perp}$, we can see $\gamma$ as the Gaussian action $\widehat{\pi|_{K}}: G\curvearrowright \widehat{K}$, the skew product action $G\curvearrowright \widehat{K^{\perp}}\times \R$ and the action $G\curvearrowright Z$. Take any infinite cyclic discrete subgroup $\Lambda$ of $G$. Then by theorem $9.1.2$ of \cite{amine1}, the action $\Lambda\curvearrowright \widehat{K^{\perp}}\times \R \times Z$ is conservative. Hence $\Lambda$ is recurrent with respect to this action. Moreover by assumption, the action $G\curvearrowright \hat{K}$ is mixing. Then $\Lambda$ is mixing with respect to it. By theorem \ref{7.13 Amine}, we can conclude that 
\begin{equation}\label{equation 4.11}L^{\infty}(\hat{H}\times \R\times Z)^{G}=1\otimes L^{\infty}(\widehat{K^{\perp}}\times \R\times Z)^G.
\end{equation}
Therefore we are in the same situation of the proof in \cite{amine1}. We can then conclude in the same way. By our assumption $c$ is evanescent. Thus there exists a decreasing sequence of closed $\pi(G)$-invariant subspaces $H_n\subset H$ and vectors $\xi_n\in H_n^{\perp}$ such that $$c(g)+\pi(g)\xi_n-\xi_n\in H_n,$$
for all $g\in G$ and $n \in \N$. Since $H$ does not admit nonzero $\pi(G)$-invariant vectors, the vectors $\xi_n$ are uniquely determined and $P_{H_n^{\perp}}(\xi_m)=\xi_n$ for all $m\ge n$. Since $c$ is not a coboundary, we have that $\|\xi_n\|\to +\infty$ when $n$ tends to infinity. Using (\ref{equation 4.11})  we get that for every $n$:
$$L^{\infty}(\hat{H}\times \R\times Z)^{G}\subset \mathcal{M}(\xi_n+H_n)\overline{\otimes} L^{\infty}(Z).$$
With the notations of proposition \ref{2.7 Amine}, we have that $M=\emptyset$ and therefore $A=1$. So we get that
$$L^{\infty}(\hat{H}\times \R\times Z)^{G}\subset 1\otimes L^{\infty}(Z)^G=\C 1.$$
So $\beta$ is weakly mixing.
\end{proof}
We end this section with the following remark which concludes the proof of the ergodicity.
\begin{remark}
A locally compact group, non-compact group $G$ admits an evanescent affine isometric action without fixed point whose linear part is mixing if and only if it has the Haagerup property. Indeed this follows from propositions $2.12$ and $2.13$ of \cite{amine2}.
So groups with Haagerup property is a class of groups for which, under the non-periodicity assumption, proposition \ref{weaklymixing} can be applied.
\end{remark}
\subsection{Examples of dynamical systems for groups acting properly on trees}
A typical example of groups which it is easy to construct a $C_0$-dynamical system with an invariant mean as in theorem \ref{THM DJZ 2} is the amenable groups. Indeed it suffices to consider the left regular representation (see proposition 1.4 \cite{delabie2019new} for example). The advantage of using Gaussian functor to prove this theorem is that it gives us a way to make explicit such system when we are considering groups acting properly on trees. We begin by some recalls in order to fix the notation. We refer to Chapter $2$ and Appendix $C$ of  \cite{bekka2008kazhdan} for more details.\\
Let $T=(V,E)$ a locally finite tree. We will identify $T$ with its set of vertices. We denote by $\E$ the set of oriented edges, where each geometric edge comes with two possible orientations. For $x,y\in V$, let $[x,y]$ be the set of oriented edges on the unique geodesic path from $x$ to $y$. We say that $e\in [x,y]$ is coherent if $e$ points from $x$ to $y$ and incoherent otherwise. There exists a unique embedding $\iota:T\to H$ into an affine Hilbert space $H$ such that :
\begin{enumerate}
\item $d(x,y)=\|\iota(x)- \iota(y)\|^2$;
\item $H$ is the closed affine span of $\iota(T)$.
\end{enumerate}
We will view $T$ directly as a subset of $H$.
Suppose that $G$ acts properly on $T$ by automorphisms, then by the uniqueness of the embedding, the action of $G$ extends uniquely to an affine isometric action $\alpha: G\curvearrowright H$ with $\alpha_g(\iota(x))=\iota(gx).$ Moreover as the action of $G$ on $T$ is proper, then combining proposition $8.6$ and $9.3$ of \cite{amine2}, we deduce that $\alpha$ is proper.  Hence we get:\\
\begin{example}
Let $G$ be a group acting properly on a locally finite tree $T$.
Let $\alpha: G\curvearrowright H$ be the affine isometric proper action associated. Therefore $$(\hat{H},\mu)=\left(\R, \frac{1}{\sqrt{2\pi}}\exp\left(-\frac{x^2}{2}\right)dx \right)^{\otimes T}.$$
This comes from the standard construction, thanks to property $(2)$ of the embedding, which can be found in \cite{cherix2012groups} Chapter $2$. Thus for the free group on two generators, $\mathbb{F}_2$, let $\alpha: \mathbb{F}_2 \curvearrowright H$ be the affine isometric action associated to its action on its Cayley graph. Then an example of infinite measure space as in theorem \ref{THM DJZ} is given by : 
$$\left(\R, \frac{1}{\sqrt{2\pi}}\exp\left(-\frac{x^2}{2}\right)dx \right)^{\otimes \mathbb{F}_2} \otimes (\R, \lambda).$$
This answers to a question from Y.Stalder.
\end{example} 
For these groups, it is possible to give an alternative way to find such a measure space. 
Consider $\Omega(T)$ the compact space of all possible orientations of $T$. Recall that an orientation of $T$ is a map $\omega : E(T)\to \left\{-1,1\right\}$. Then we define for every $x,y\in T$ the continuous function $c(x,y)$ on $\Omega(T)$ by the formula:
$$c(x,y)(\omega)=\sum_{e\in [x,y]} \omega(e)=n_e-m_e,$$
where $n_e$ (respectively $m_e$) is the number of edges of $[x,y]$ which is coherent (respectively incoherent) by the orientation $\omega\in \Omega(T)$. Observe that we have the cocycle relation, i.e. $c(x,y)+c(y,z)=c(x,z)$ for all $x,y,z\in T$. Fix a base-point $x_0\in T$ and define a random orientation of $T$ as follows: each edge $e\in E(T)$ independently is oriented towards $x_0$ with probability $p$ and away from $x_0$ with probability $1-p$. This defines a probability measure on $\Omega(T)$ denoted $\mu_{x_0}^p$. 
Thus we can see $(\Omega(T),\mu_{x_0}^p)=\prod_{e\in E(T_4)}(\left\{-1,1\right\},p\delta_1 + (1-p)\delta_{-1})$. Now we consider the cocycle $c\colon G\times \Omega \to \Z$ defined by:
$$c(g,\omega)=c(x_0,gx_0)(\omega).$$
Therefore we have $c(gh,\omega)=c(g,\omega)+c(h,g^{-1}\omega).$
Finally take the skew product action of $G$ on $\left(\Omega(T)\times \R, \mu\otimes \lambda\right)$, given by :
$$\beta_g(\omega, t)=(\omega, t+ c(g^{-1},\omega)).$$
To prove that the permutation representation is $C_0$, denote $X$ a Bernoulli random variable with values in $\left\{-1,1\right\}$ with parameter $p$ and let $(X_e)_{e\in E(T)}$ be a sequence of i.i.d. random variables, indexed by the edges, with the same distribution as $X$. Let $(g_n)$ be a sequence in $G$ which tends to infinity. For each $n$, set $$S_n:=S_{g_nx_0}=\sum_{e\in [x_0,g_nx_0]} X_e.$$
As the action of $G$ on $T$ is proper, we get a random process $(S_n)_{n\ge 1}$ which is nothing but a kind of random walk on $\Z$. Put $p\neq \frac{1}{2}$ in order to have that $S_{\infty}$ can be simulated by a Markov chain, which is a transient walk. Thus :
$$\mathbb{P}(\lim_{n\to \infty } |S_n|=+\infty)=1.$$
Hence for almost every $\omega$, the cocycle $c(g,\omega)$ is proper. Then it suffices to consider for example the vector $\xi=1\otimes h\in L^2(\Omega\times \R, \mu_{x_0}^p\otimes \lambda)$, with $h(t)=e^{-t^2}$. One has :
$$\langle \pi(g_n)\xi,\xi\rangle=\int_{\Omega\times \R} e^{-(t+c(g_n,\omega))^2}e^{-t^2}d(\mu_{x_0}^p\otimes \lambda)(\omega,t)\le C\int_{\Omega} e^{-\frac{c(g_n,\omega)^2}{2}} d\mu_{x_0^p}(\omega),$$
for some constant $C>0$. 
Using the fact that the cocycle is proper almost everywhere and dominated convergence, we deduce: $$\lim_{n\rightarrow \infty}\langle \pi(g_n)\xi,\xi\rangle=0.$$
Since the action on the second variable mimics the behavior of the translation action of $\Z$ on $\R$, then the Koopman representation $\pi$ associated to $\beta$ admits almost invariant vectors. Indeed consider for each $n$: $$\xi_n=1\otimes \frac{1}{\sqrt{2n}}\mathbf{1}_{[-n,n]},$$
then we have :
$$\langle \pi(g) \xi_n, \xi_n\rangle= \int_{\Omega\times \R} \mathbf{1}_{[-n,n]}(t+S_{gx_0}(\omega))\mathbf{1}_{[-n,n]}(t)d(\mu\otimes \lambda)(\omega,t).$$
Using Fubini-Tonelli, we get :
$$\langle \pi(g) \xi_n, \xi_n\rangle=\int_{\Omega} \frac{2n-|S_{gx_0}(\omega)|}{2n}d\mu(\omega).$$
Now let $K$ be a compact subset of $G$. It rests to apply dominated convergence to conclude that $$\sup_{g\in K} \langle \pi(g) \xi_n, \xi_n\rangle\xrightarrow{n\to \infty} 1.$$
Hence for the free group $\mathbb{F}_2$, we obtain:
\begin{example}\label{example groupe libre}
Let $\sigma:\mathbb{F}_2\curvearrowright T_4$ its action on its Cayley graph. Then an explicit example of infinite measure space in the hypothesis of theorem  \cite{delabie2019new} is given by:
$$(\Omega(T_4)\times \R, \mu \otimes \lambda).$$
\end{example}

\bibliographystyle{alpha}
\bibliography{biblioVNA}

\end{document}